\documentclass[a4paper,11pt]{amsart}         

\makeatletter
\@namedef{subjclassname@2020}{%
	\textup{2020} Mathematics Subject Classification}
\makeatother

\usepackage{amsfonts,amsmath,latexsym,amssymb} 
\usepackage{amsthm}                
\usepackage{mathrsfs,upref}         
\usepackage{mathptmx}		    

\usepackage[english]{babel}
\usepackage{enumitem}
\setlist[itemize,enumerate,description]{itemsep=0pt,topsep=0pt}
\usepackage{mathtools}

\theoremstyle{plain}
\newtheorem{theorem}{Theorem}[section]
\newtheorem{lemma}[theorem]{Lemma}

\theoremstyle{definition}
\newtheorem{qu}{Question} 

\numberwithin{equation}{section}

\newcommand{\N}{\mathbb{N}} 
\newcommand{\C}{\mathbb{C}} 
\newcommand{\R}{\mathbb{R}} 
\DeclareMathOperator{\udens}{\overline{dens}} 
\DeclareMathOperator{\ldens}{\underline{dens}} 
\newcommand{\abs}[1]{\left| #1 \right|}
\newcommand{\norm}[1]{{\left\|#1\right\|}}
\newcommand{\harmsp}{\mathcal{H}(\mathbb{R}^N)}
\newcommand{\harmpolys}[1]{\harmpolysArgs{#1}{N}}
\newcommand{\harmpolysArgs}[2]{\mathcal{H}_{#1}(\mathbb{R}^{#2})}
\newcommand{\Rplus}{\mathbb{R}_+}
\newcommand{\eps}{\varepsilon}

\title{Rate of Growth of Distributionally Chaotic Functions}

\author{C. Gilmore}

\address{
	School of Mathematical Sciences\\
	University College Cork\\ 
	Ireland
}

\email{clifford.gilmore@ucc.ie}

\author{F. Mart\'{\i}nez-Gim\'{e}nez}

\address{
	Institut Universitari de Matemàtica Pura i Aplicada\\
	Universitat Polit\`ecnica de Val\`encia\\
	Edifici 8E, Acces F, 4a planta\\
	46022 Val\`encia \\
	Spain
}

\email{fmarting@mat.upv.es}

\author{A. Peris}

\address{
	Institut Universitari de Matemàtica Pura i Aplicada\\
	Universitat Polit\`ecnica de Val\`encia\\
	Edifici 8E, Acces F, 4a planta\\
	46022 Val\`encia \\
	Spain
}

\email{aperis@mat.upv.es}

\date{}                               





\thanks{C. Gilmore was supported by the Magnus Ehrnrooth Foundation and the Irish Research Council via a Government of Ireland Postdoctoral Fellowship. F. Mart\'{\i}nez-Gim\'{e}nez and A. Peris were supported by MCIN/AEI/10.13039/501100011033, Project PID2019-105011GB-I00, and by Generalitat Valenciana, Projects PROMETEO/2017/102 and PROMETEU/2021/070}

\begin{document}

\begin{abstract}
We investigate the permissible growth rates of functions that are distributionally chaotic with respect to differentiation operators.  We improve on the known growth estimates for $D$-distributionally chaotic entire functions, where growth is in terms of average $L^p$-norms on spheres of radius $r>0$ as $r \to \infty$, for $1 \leq p \leq \infty$. We compute  growth estimates of $\partial/ \partial x_k$-distributionally chaotic harmonic functions in terms of the average $L^2$-norm on spheres of radius $r>0$ as $r \to \infty$. We also calculate sup-norm growth estimates of distributionally chaotic harmonic functions in the case of the partial differentiation operators $D^\alpha$.
\end{abstract}

\maketitle



\section{Introduction}
\label{intro}

The term \emph{chaos} first appeared in mathematical literature in an article by Li and Yorke~\cite{LY75}, where they studied the dynamical behaviour of interval maps with period
three.   Schweizer and Sm{\'\i}tal~\cite{SS94} subsequently introduced the stronger notion of \emph{distributional chaos} for self-maps of a compact interval.

A continuous map $g \colon Y \to Y$ on a metric space $(Y,d)$ is said to be
\emph{ Li-Yorke chaotic} if there exists an uncountable set $\Gamma \subset Y$
such that for every pair $(x,y) \in \Gamma \times \Gamma$ of distinct points,
we have
\begin{equation*}
\liminf_{n \to \infty} d\left(g^n (x),\, g^n (y)\right) = 0
\ \ \mbox{ and } \ \
\limsup_{n \to \infty} d\left(g^n (x),\, g^n (y)\right) > 0.
\end{equation*}
In this case, $\Gamma$ is called a \emph{scrambled set} for $g$ and each
such pair $(x,y)$ is called a \emph{Li-Yorke pair} for $g$.
This definition captures the behaviour of orbits which are proximal without being asymptotic.

Our setting will be a Fr\'{e}chet space $X$ endowed with an increasing sequence $\left( \norm{ \; \cdot \; }_k \right)_{k \in \N}$ of seminorms that define the metric
\begin{equation*}
d(x,y) \coloneqq \sum_{k=1}^\infty 2^{-k} \frac{\norm{x-y}_k}{1+\norm{x-y}_k}
\end{equation*}
under which $X$ is complete, where $x, y \in X$.  We let $T$ denote a continuous linear operator on $X$.

The connection between Li-Yorke chaos and the linear dynamical property of irregularity was identified in \cite{BBMGP11}.
We say that $x \in X$ is an \emph{irregular vector} for $T$  if there exist $m \in \N$ and increasing sequences $(j_k)$ and $(n_k)$ of positive integers such that
\begin{equation*}
\lim_{k \to \infty} T^{j_k} x = 0  \quad \textnormal{ and } \quad \lim_{k \to \infty} \norm{ T^{n_k} x }_m = \infty.
\end{equation*}
This notion was introduced by Beauzamy~\cite{Bea88} for Banach spaces to describe the local aspects of the dynamics of pairs of vectors and it was generalised to  the Fr\'{e}chet
space setting  in \cite{BBMP15}.

We recall if there exists $x \in X$  such that its $T$-orbit is dense in $X$, that is
\begin{equation*}
\overline{\left\lbrace T^n x : n \geq 0 \right\rbrace} = X,
\end{equation*}
then $T$ is said to be \emph{hypercyclic} and such an $x \in X$ is known as a \emph{hypercyclic vector}.
It is well known that hypercyclic vectors are irregular.  Comprehensive introductions to the topic of hypercyclicity can be found in the monographs \cite{BM09} and \cite{GEP11}, and a survey of some recent advances in the area can be found in \cite{Gil20}.

A natural strengthening of Li-Yorke chaos was introduced by Schweizer and Sm\'ital \cite{SS94} with the notion of distributional chaos for interval maps. We first recall that the upper and lower densities of a set $A \subset \N$ are defined, respectively, as
\begin{align*}
\udens(A) &\coloneqq \limsup_{n \to \infty} \frac{ \left| A \cap \{ 1, 2, \dotsc , n\} \right|}{n}, \\
\ldens(A) &\coloneqq \liminf_{n \to \infty} \frac{ \left| A \cap \{ 1, 2, \dotsc , n\} \right|}{n}.
\end{align*}
For a continuous map $g \colon Y\to Y$ on a metric space $(Y,d)$, points $x ,y \in Y$ and
$\delta > 0$, we define
\begin{align*}
F_{x,y}(\delta) &\coloneqq \ldens(\{n \in \N \, : \, d\left(g^n (x),\, g^n (y)\right) < \delta\})
\shortintertext{and}
F_{x,y}^*(\delta) &\coloneqq \udens(\{n \in \N \, : \, d\left(g^n (x),\, g^n (y)\right) < \delta\}).
\end{align*}
If the pair $(x,y)$ satisfy $F^*_{x,y} \equiv 1$ and $F_{x,y}(\eps) = 0$
for some $\eps > 0$,
then $(x,y)$ is called a \emph{distributionally chaotic pair}. The map $g$ is said to be
\emph{distributionally chaotic}  if there exists an
uncountable set $\Gamma \subset Y$ such that every distinct pair $(x,y) \in \Gamma \times \Gamma$ is a distributionally chaotic pair  for $g$.

The study of distributional chaos in the linear dynamical setting was initiated in \cite{MGOP09} and it is intrinsically connected to the following properties.
The $T$-orbit of $x \in X$ is said to be \emph{distributionally near to 0}  if there exists $A \subset \N$ with $\udens(A) = 1$  such that
\begin{equation*}
\lim_{n \in A} T^n x = 0.
\end{equation*}
We say $x \in X$ has a \emph{distributionally unbounded orbit} if there exist $m \in \N$ and $B \subset \N$ with $\udens(B) = 1$  such that
\begin{equation*}
 \lim_{n \in B} \norm{ T^n x }_m = \infty.
\end{equation*}
Combining these properties, $x \in X$ is defined to be a \emph{distributionally irregular vector} for $T$ if its orbit is both distributionally unbounded and distributionally near
to 0.  This strengthening of irregularity was introduced in \cite{BBMGP11}.

In the Fr\'{e}chet space setting, it follows from results in \cite{BBMGP11} and \cite{BBMP13} that $T$ admits a distributionally irregular vector if and only if  $T$ is
distributionally chaotic. Hence, in the sequel our study focuses on distributionally irregular vectors.
Distributional chaos has been investigated from many aspects, for instance in \cite{ABMP13,BR15,BBMP15,BBPW18,CKMM16,MGOP13,WC16,WWC17,WZ12a,YY18,BBP20}.

An example of a map that admits a distributionally irregular vector is the  differentiation operator $D \colon f \mapsto f'$, acting on the  space $H(\C)$ of entire
functions on $\C$ (this follows from \cite[Corollary 17]{BBMP13}).
Bernal and Bonilla~\cite{BGB16} computed growth estimates for $D$-irregular and $D$-distributionally irregular entire functions, where growth is in terms of average $L^p$-norms, for
$1 \leq p \leq \infty$, on spheres of radius $r>0$ as $r \to \infty$.

We note that permissible growth rates of $D$-hypercyclic and $D$-frequently hypercyclic entire functions have previously been investigated in
\cite{BBGE10,BB13,DS12,DR84,Gro90,Mac52,Nik14,Shk93}.  We recall the notion of frequent hypercyclicity was introduced by Bayart and Grivaux~\cite{BG06}, where they defined
$T \colon X \to X$ to be \emph{frequently hypercyclic} if there exists $x \in X$ such that for any nonempty open subset $U \subset X$ it holds that
\begin{equation*}
\ldens \left( \left\{ n : T^n x \in U \right\}\right)  > 0.
\end{equation*}
Such an $x \in X$ is  called a \emph{frequently hypercyclic vector} for $T$.

In the setting of the space $ \harmsp$ of  harmonic functions  on $\R^N$, for $N \geq 2$, Aldred and Armitage~\cite{AA98a} considered the linear dynamical properties of partial differentiation
operators
\begin{equation*}
\frac{\partial}{\partial x_k} \colon \harmsp \to \harmsp,
\end{equation*}
where  $1 \leq k \leq N$.
They identified sharp $L^2$-growth rates, on spheres of radius $r>0$ as $r \to \infty$, of harmonic functions that are  universal (and hence hypercyclic) with respect to $\partial/
\partial x_k$.
Growth estimates in the frequently hypercyclic case were computed by Blasco et al.~\cite{BBGE10}   and sharp growth rates were subsequently identified in \cite{GST19}.

Growth estimates with respect to the  sup-norm, on spheres of radius $r>0$ as $r \to \infty$,  were computed by Aldred and Armitage~\cite{AA98b} for harmonic functions that are
universal (and hence hypercyclic) for general partial differentiation operators
\begin{equation*}
D^\alpha = \frac{\partial^{\abs{\alpha}}}{\partial x_1^{\alpha_1} \cdots \partial x_N^{\alpha_N}}
\end{equation*}
where $\alpha = (\alpha_1, \dotsc, \alpha_N) \in \N^N$ and $|\alpha| = \alpha_1 + \cdots + \alpha_N$.
 The frequently hypercyclic case was subsequently investigated by Blasco et al.~\cite{BBGE10}.

In this article we compute permissible growth rates of irregular and distributionally irregular functions.
In Section \ref{sec:growthEntireFns} we improve the growth estimates from \cite{BGB16}  for distributionally irregular entire functions and we also provide lower estimates.  In
Section \ref{sec:growthHarmonicFnsI} we investigate average $L^2$-growth estimates of irregular and distributionally irregular harmonic functions with respect to partial
differentiation operators $\partial/ \partial x_k$.  Then in Section \ref{sec:growthHarmonicFnsII} we compute sup-norm growth rates of distributionally irregular harmonic functions
in the case of the partial differentiation operators $D^\alpha$.

\section{Preliminaries}

In this section we collect some results that are required in the sequel.

We recall that absolutely Ces\`{a}ro bounded operators cannot be distributionally irregular.
For a Banach space $X$, the continuous linear operator $T \colon X \to X$ is said to be \emph{absolutely Ces\`{a}ro bounded} if there exists a constant $C>0$ such that
\begin{equation*}
\sup_{N \in \N} \frac{1}{N}  \sum_{j=1}^N \norm{T^j x} \leq C \norm{x}
\end{equation*}
for all $x \in X$.  If the orbit of $x \in X$ is distributionally unbounded, then it was proven in \cite[Proposition 20]{BBPW18}  that
\begin{equation*}
\limsup_{N \to \infty} \frac{1}{N} \sum_{j=0}^N \norm{T^j x} = \infty.
\end{equation*}

The following estimate will be needed, it can be found in \cite[Lemma 2.2]{BBGE10}.
\begin{lemma}\label{lma:BarnesEstimate}
	Let $0 < \alpha \leq 2$ and $\beta \in \R$. Then there exists some constant $C>0$ such that, for all $r>0$,
	\begin{equation*}
		\sum_{n=0}^\infty \frac{r^{\alpha n}}{(n+1)^\beta n!^\alpha} \leq C \frac{e^{\alpha r}}{r^{(\alpha + 2\beta -1)/2}}.
	\end{equation*}
\end{lemma}

We will need a technical result whose proof follows the argument of \cite[Theorem 2.4]{BBGE10}.

\begin{lemma}\label{keyl}
Let $\alpha,\beta>0$. Then there exists some $M>0$ such that, for any non-negative sequence  $(x_n)$,
\begin{equation}\label{boundser}
 \sup_{m\geq 1} \frac{1}{m} \sum_{n=1}^mx_n \leq M \sup_{R>0}    \sum_{n=1}^\infty x_n \frac{R^{\alpha n+\beta}e^{-\alpha R}}{n!^\alpha n^{\beta-\alpha/2+1/2}} .
\end{equation}
\end{lemma}

\begin{proof}
We consider the functions
\[
g_n(R) \coloneqq \frac{R^{\alpha n+\beta}e^{-\alpha R}}{n!^\alpha n^{\beta-\alpha/2+1/2}}, \ \ R>0, \ \ n\in\N .
\]
The function $g_n$ attains its maximum at $a_n \coloneqq n+\beta/\alpha$. Moreover by Stirling's formula
\[
g_n(a_n) \coloneqq \frac{(n+\beta/\alpha)^{\alpha n+\beta}e^{-\alpha {(n+\beta/\alpha)}}}{n!^\alpha n^\beta n^{-\alpha/2}n^{1/2}} \sim \frac{1}{\sqrt{n}}.
\]
The function $g_n$ has an inflection point at $b_n \coloneqq a_n+\sqrt{n/\alpha +\beta/\alpha^2}$ and, since $\lim_{R\to\infty}g_n(R)=0$, we have that $g_n(R)\geq h_n(R)$ for each $R\in I_n = [a_n,b_n]$, where $h_n$ is the affine map such that $h_n(a_n)=g_n(a_n)$ and $h_n(b_n)=0$. We fix $m_0\in\N$ such that
\[
m<a_n=n+\beta/\alpha<b_n=a_n+\sqrt{n/\alpha+\beta/\alpha^2}<3m
\]
for $m\geq m_0$ and for $n=m,\dots ,2m$. We have that
\begin{align*}
 2m\left( \sup_{R>0}    \sum_{n=1}^\infty x_ng_n(R)\right)
 &\geq \int_{[m,3m]} \left( \sum_{n=m}^{2m} x_ng_n(s)\right)ds
 \geq \sum_{n=m}^{2m} x_n  \int_{I_n} g_n(s)ds\\
 &\geq \sum_{n=m}^{2m} x_n \int_{I_n} h_n(s)ds
 \geq C' \sum_{n=m}^{2m} x_n \frac{\sqrt{n/\alpha+\beta/\alpha^2}}{\sqrt{n}}
 \geq \frac{C'}{\sqrt{\alpha}} \sum_{n=m}^{2m} x_n
\end{align*}
for $m\geq m_0$, where $C'>0$ is a constant independent of $m$ and $(x_n)$. Therefore we find $M_1>0$ with
\[
\frac{1}{m} \sum_{n=m}^{2m} x_n \leq M_1 \sup_{R>0}    \sum_{n=1}^\infty x_n \frac{R^{\alpha n+\beta}e^{-\alpha R}}{n!^\alpha n^{\beta-\alpha/2+1/2}}, \ \ \forall m\geq m_0.
\]
On the other hand,
\[
\sum_{n=m_0}^m \left( \frac{1}{n} \sum_{j=n+1}^{2n} x_j\right) \geq \sum_{j=m_0+1}^m x_j\left( \sum_{\max\{m_0,j/2\} \leq n<j} \frac{1}{n} \right)
\geq \frac{1}{m_0} \sum_{j=m_0+1}^m x_j,
\]
so we get
\begin{align*}
\frac{1}{m} \left( \sum_{j=m_0+1}^m x_j\right) &\leq \frac{m_0}{m} \sum_{n=m_0}^m \left( \frac{1}{n} \sum_{j=n+1}^{2n} x_j\right) \\
&\leq m_0M_1 \sup_{R>0}    \sum_{n=1}^\infty x_n \frac{R^{\alpha n+\beta}e^{-\alpha R}}{n!^\alpha n^{\beta-\alpha/2+1/2}}
\end{align*}

for all $m\geq m_0$, and we find $M>0$ independent of  $(x_n)$ with
\[
\sup_{m\geq 1} \frac{1}{m} \sum_{n=1}^mx_n \leq M \sup_{R>0}    \sum_{n=1}^\infty x_n \frac{R^{\alpha n+\beta}e^{-\alpha R}}{n!^\alpha n^{\beta-\alpha/2+1/2}} .
\]
\end{proof}

The following is a slight improvement of a result that is a consequence of Proposition~7 and Theorems~15 and 19 in \cite{BBMP13}.  For the
convenience of the reader we outline the main steps of the proof.

\begin{theorem}[Bernardes et al.~\cite{BBMP13}] \label{thm:sufficient}
	Let $X$ be a Fr\'{e}chet space, let $T \colon X \to X$ be a continuous linear operator, and assume that $Y$ is a separable Fr\'{e}chet space that is continuously
embedded in $X$. Suppose that:
	\begin{enumerate}[label=(\alph*)]
		\item There exists a dense subset $Y_0$ of $Y$ with $T^n(Y_0)\subset Y$ for each $n\in\N$ and
		$\lim_{n \to \infty} T^ny = 0$ in $Y$ for all $y \in Y_0$.  \label{item:suffA}

		\item There exist a subset $Y_1$ of $Y$, a map $S \colon Y_1 \to Y_1$
		with $TSy = y$ on $Y_1$, and a vector $z \in Y_1 \setminus \{0\}$ such that
		$\sum_{n=1}^\infty T^nz$ and $\sum_{n=1}^\infty S^nz$
		converge unconditionally in $X$ and $Y$, respectively.  \label{item:suffB}
	\end{enumerate}
	Then there exists a dense subset of vectors in $Y$ which are distributionally irregular for $T \colon X \to X$.
\end{theorem}

\begin{proof}
	We let $\left( \norm{ \; \cdot \; }_k \right)_{k \in \N}$ denote an increasing fundamental sequence of seminorms on $X$, and without loss of generality we assume it satisfies
$$
\norm{Tx}_k \leq \norm{x}_{k+1} \  \mbox{ for all } x \in X \mbox{ and }
 k \in \N.
$$
We  fix another increasing fundamental sequence $\left( \norm{ \; \cdot \; }_k' \right)_{k \in \N}$ of seminorms in $Y$, with $\norm{y}_k\leq \norm{y}_k'$ for every $y\in Y$.

	Arguing as in \cite[Theorem 19]{BBMP13}, we define $w_{k_0} \coloneqq \sum_{n=1}^\infty T^{k_0n}z + z + \sum_{n=1}^\infty S^{k_0n}z$, where $w_{k_0} \neq 0$ if $k_0$ is
sufficiently large and $T^{k_0}w_{k_0} = w_{k_0}$.  Let $y_k \coloneqq \sum_{n=k}^\infty S^{k_0n}z\in Y$, $k\in\N$.  Then $y_k \to 0$ in $Y$ and
	\begin{equation*}
	T^{k_0j}y_k = \sum_{n=1}^{j-k} T^{k_0n}z + z + \sum_{n=1}^\infty S^{k_0n}z \to w_{k_0}
	\end{equation*}
	in $X$ as $j \to \infty$.  For $0 \leq \ell < k_0$ we have
	\begin{equation*}
	\lim_{j \to \infty} T^{\ell + k_0j} y_k = T^\ell w_{k_0} \ \mbox{ in } X
	\end{equation*}
	and hence $\{ T^\ell w_{k_0} : 0 \leq \ell < k_0 \}$ are accumulation points of the orbit of $y_k$.

	Let $m\in\N$ with $\norm{T^\ell w_{k_0}}_m\neq 0$, $\ell=0,\dots ,k_0-1$. We define
	\begin{equation*}
	\varepsilon \coloneqq \frac{1}{2} \min \left\lbrace \norm{ \; T^\ell w_{k_0}  \; }_m \ : \ 0 \leq \ell < k_0 \right\rbrace >0.
	\end{equation*}
	Then there exists an increasing sequence $(N_k)$ of positive integers such that
	\begin{equation}\label{unbY}
	\lim_{k\to \infty} \frac{1}{N_k} \abs{\left\lbrace 1 \leq j \leq N_k :  \norm{ \; T^j y_k  \; }_m > \varepsilon \right\rbrace } = 1.
	\end{equation}
	Without loss of generality we will assume that $m=1$. We adapt now the proof of \cite[Theorem 15]{BBMP13} to show that there exists a dense subset of vectors in $Y$ which are distributionally irregular for $T \colon X \to X$. Indeed,  by~\ref{unbY} and since the sequences $(y_k)$ and $(T^iu)$, $u\in Y_0$, converge to $0$ in $Y$, we can construct inductively a sequence $(x_k)$ of vectors in $Y_0$
with $\norm{x_k}'_k \leq 1$, $k \in \N$, and an increasing sequence $(n_k)$
of positive integers such that
\begin{align}
 &  \abs{\left\lbrace 1 \leq i \leq n_k : \|T^ix_k\|_1 > k2^k \right\rbrace } >
  n_k \left(1 - \frac{1}{k^2}\right), \label{eq1}
 \\
 & \abs{\left\lbrace 1 \leq i \leq n_k : \|T^ix_s\|'_k < \frac{1}{k} \right\rbrace } >
  n_k \left(1 - \frac{1}{k^2}\right), \ s = 1,\dots,k-1. \label{eq2}
\end{align}

We consider an increasing sequence
$(r_j)$ of positive integers such that
\begin{equation}
r_{j+1} \geq 1 + r_j + n_{r_j + 1} \ \ \text{ for all } j \in \N. \label{eq3}
\end{equation}
We fix $\alpha \in \{0,1\}^\N$ defined by $\alpha_n = 1$ if and only if
$n = r_j$ for some $j \in \N$, and we define the vector
\[
 u := \sum_i \frac{\alpha_i}{2^i}\, x_i
          = \sum_j \frac{\alpha_{r_j}}{2^{r_j}}\, x_{r_j}
\]
{with $u \in Y$} by the fact that $\norm{x_k}'_k \leq 1$, $k \in \N$. The argument of the proof of \cite[Theorem 15]{BBMP13} yields that $u\in Y$ is a distributionally irregular vector for $T \colon X \to X$ since $\norm{y}_k\leq \norm{y}_k'$ for every $y\in Y$.

Finally, the set $u+Y_0$ forms a dense subset of vectors in $Y$ which are distributionally irregular for $T \colon X \to X$, as desired.
\end{proof}

\section{Growth of Distributionally Irregular Entire Functions}  \label{sec:growthEntireFns}

In this section we consider growth of entire functions that are distributionally irregular with respect to the  differentiation operator $D$ acting on the space $H(\C)$ of entire functions.  We first introduce the pertinent definitions that allow us to precisely specify how growth is measured.

For an entire function $f \in H(\C)$ and $1 \leq p < \infty$, the average $L^p$-norm is defined as
\begin{equation*}
M_p (f, r) = \left( \frac{1}{2 \pi} \int_{0}^{2\pi} \lvert f(re^{it}) \rvert^p \; dt \right)^{1/p}
\end{equation*}
where $r > 0$ and we denote the sup-norm of $f$ by
\begin{equation*}
M_\infty (f,r)  = \sup_{\abs{z} = r} \abs{f(z)}, \quad r > 0.
\end{equation*}

Optimal growth rates of $D$-hypercyclic entire functions were identified by Grosse-Erdmann~\cite{Gro90} and Shkarin~\cite{Shk93}. They proved if $\varphi \colon \R_+ \to \R_+$ is a function such that $\varphi(r) \to \infty$ as $r \to \infty$, then there exists a $D$-hypercyclic entire function $f \in H(\C)$ such that
\begin{equation} \label{entireCriticalGrowth}
M_\infty (f,r) \leq \varphi(r) \frac{e^r}{\sqrt{r}}
\end{equation}
for $r$ sufficiently large. This growth is optimal, since for the critical rate of $e^r/\sqrt{r}$, there does not exist a $D$-hypercyclic entire function $f \in H(\C)$ such that
\begin{equation} \label{entireCriticalGrowthLower}
M_\infty (f,r)\leq c \frac{e^r}{\sqrt{r}}, \quad \textrm{for }  r >0,
\end{equation}
where $c>0$ is a constant. As noted by Blasco et al.~\cite[Theorem 2.1]{BBGE10}, the above growth results extend to $M_p$-averages for all $1 \leq p \leq \infty$.

As expected, $D$-frequently hypercyclic entire functions must grow faster than in the hypercyclic case.  Permissible growth  of $D$-frequently hypercyclic entire functions was investigated in \cite{BBGE10}, and optimal growth was identified by Drasin and Saksman~\cite{DS12}. In \cite{DS12} they proved for any constant $C >0$, that there exists a $D$-frequently hypercyclic function $f \in H(\C)$  such that
	\begin{equation*}
	M_\infty (f,r) \leq C \frac{e^r}{r^{1/4}}, \quad \textrm{ for all } r >0.
	\end{equation*}
The above growth result naturally applies to $M_p$-averages for $1 < p \leq \infty$. However, we note that in the case $p=1$, Bonet and Bonilla~\cite{BB13} had previously identified that there exists a $D$-frequently hypercyclic entire function $f \in H(\C)$ with
\begin{equation*}
	M_1(f,r) \leq \varphi(r) \frac{e^r}{r^{1/2}}
\end{equation*}
for all $r>0$ and where $\varphi \colon \R_+ \to \R_+$ is any function such that $\varphi(r) \to \infty$ as $r \to \infty$.

Bernal and Bonilla~\cite{BGB16} subsequently identified optimal growth estimates for $D$-irregular entire functions. They proved that the critical rate of growth for $D$-irregular  entire functions is the same as given by \eqref{entireCriticalGrowth} and \eqref{entireCriticalGrowthLower} in the  $D$-hypercyclic case.

Initial growth estimates for $D$-distributionally irregular functions were also obtained in \cite{BGB16}.
They proved for $1 \leq p \leq \infty$ and $a = \left( 2 \max\{ 2, p\} \right)^{-1}$, that for every $\varepsilon > 0$  there exists a $D$-distributionally irregular function $f \in H(\C)$ such that
\begin{equation*}
	 M_p (f,r) \leq C \frac{e^r}{r^{a - \varepsilon}}
\end{equation*}
for some constant $C>0$.

Here we improve the permissible growth estimates for $D$-distributionally irregular entire functions, while also computing estimates for lower rates of growth.

\begin{theorem} \label{thm:DIgrowthEntire}
	Let $1 \leq p \leq \infty$.
	\begin{enumerate}[label=(\roman*)]
		\item  		Let $a = 1/(2 \max \left\lbrace 2, p \right\rbrace)$.  For any $\varphi \colon \Rplus \to \Rplus$ with $\varphi(r) \to \infty$ as $r \to \infty$, there exists
a $D$-distributionally irregular entire function $f$ with
		\begin{equation*}
		M_p(f,r) \leq \varphi(r) \frac{e^r}{r^a}
		\end{equation*}
		for $r>0$ sufficiently large.  \label{item:EntireGrowthI}

			\item
		Let $a = 1/(2 \min \left\lbrace 2, p \right\rbrace)$.  There does not exist a $D$-distributionally irregular entire function $f$ that satisfies
		\begin{equation}  \label{ineq:lowerGrowthEntire}
		M_p (f, r) \leq c \frac{e^r}{r^a}
		\end{equation}
		where $c >0$ is a constant and for $r > 0$ sufficiently large.  \label{item:EntireGrowthII}
	\end{enumerate}
\end{theorem}

\begin{proof}
\ref{item:EntireGrowthI}
	Since
	\begin{equation*}
		M_p (f, r) \leq M_2 (f, r)
	\end{equation*}
	for $1 \leq p < 2$ we need only prove the result for $p \geq 2$.
	Let $2 \leq p \leq \infty$ and we assume without loss of generality that $\inf_{r>0}\varphi (r)>0$.
		Our strategy is to apply Theorem \ref{thm:sufficient}, so we consider
    \begin{equation*}
	Y \coloneqq \left\lbrace  f \in H(\C)  \ :  \ \lim_{r\to \infty} \frac{ M_p (f ,r)r^{1/2p}}{\varphi (r)e^r} =0 \right\rbrace
	\end{equation*}
	endowed with the natural sup-norm {which we denote by $\norm{\, \cdot \, }_Y$.}  We note that $\left( Y, \norm{\, \cdot \, }_Y\right)$  is Banach space which is continuously embedded in $H(\C)$.  Further note that the functions $f \in Y$ satisfy the
desired growth condition.

We let $Y_0 = Y_1$ be the space of polynomials and we note that $Y_0$ is dense in $Y$. It follows immediately that Theorem \ref{thm:sufficient} \ref{item:suffA} is satisfied.

	Next we check Theorem \ref{thm:sufficient} \ref{item:suffB}. We define the mapping $S \colon Y_1 \to Y_1$ for $g \in Y_1$ as
	\begin{equation*}
		S g(z) = \int_0^z g(\xi) \, \mathrm{d}\xi.
	\end{equation*}
	For all $g \in Y_1$  we have that $DSg = g$, and since $\sum_{n=1}^\infty D^n g$ is a finite series it converges unconditionally.

	It remains to show that $\sum_{n=1}^\infty S^n g$ converges unconditionally in $Y$ for any polynomial $g$.  (This calculation appears in \cite[Theorem 2.3]{BBGE10}, but for the convenience of the reader we recall the argument here.)

	It suffices to consider monomials $g(z) = z^k$, $k \in \N$, which gives
	\begin{equation*}
		\sum_{n=1}^\infty S^n g(z) = \sum_{n=1}^\infty \frac{k!}{(k+n)!} z^{k+n}
	\end{equation*}
	and thus we need to prove that $\sum_{n=1}^\infty z^n/n!$ converges unconditionally in $Y$.

	To this end, let $\varepsilon >0$ and $N \in \N$. By the Hausdorff-Young Inequality (cf.~\cite{Kat76}), we obtain for any finite subset $F \subset \N$ that
	\begin{equation*}
		M_p \left( \sum_{n \in F} \frac{z^n}{n!}, r \right) \leq \left( \sum_{n \in F} \frac{r^{qn}}{n!^q} \right)^{1/q},
	\end{equation*}
	where $q$ is the conjugate exponent of $p$, i.e.\ $1/p + 1/q = 1$.
	Hence, if $F \cap \{ 0,1, \dotsc, N \} = \varnothing$, then
	\begin{equation*}
		\norm{\sum_{n \in F} \frac{z^n}{n!}}_Y \leq  \left( \sup_{r>0} \frac{r^{q/2p}}{\varphi(r)^q e^{qr}} \sum_{n > N} \frac{r^{qn}}{n!^q} \right)^{1/q}.
	\end{equation*}
	We choose $R>0$ such that $\varphi(r)^q \geq 1/\varepsilon$ for $r \geq R$. Then it follows that
	\begin{equation*}
		\sup_{r \leq R} \frac{r^{q/2p}}{\varphi(r)^q e^{qr}} \sum_{n > N} \frac{r^{qn}}{n!^q} \leq \frac{R^{q/2p}}{ \inf_{r>0} \varphi(r)^q } \sum_{n > N} \frac{R^{qn}}{n!^q} \to 0
	\end{equation*}
	as $N \to \infty$. Moreover, Lemma \ref{lma:BarnesEstimate} gives that
	\begin{equation*}
		\sup_{r \geq R} \frac{1}{\varphi(r)^q} \frac{r^{q/2p}}{e^{qr}} \sum_{n > N} \frac{r^{qn}}{n!^q} \leq C \varepsilon
	\end{equation*}
	for any $N \in \N$, where $C$ is a constant depending only on $q$.
	Thus,
	\begin{equation*}
	\norm{\sum_{n \in F} \frac{z^n}{n!}}_Y^q \leq  (1+C)\varepsilon
	\end{equation*}
	if $\min F > N$ and $N$ is sufficiently large, so that $\sum_{n=1}^\infty z^n/n!$ converges unconditionally in $Y$.

\noindent\ref{item:EntireGrowthII}
For $p=1$ the result follows from \cite[Theorem 7]{BGB16}.
Moreover since
\begin{equation*}
M_2(f,r) \leq M_p(f,r)
\end{equation*}
for $2 < p \leq \infty$, it suffices to prove the result for $p \leq 2$.
So we assume  $1 < p \leq 2$ and that $f \in H(\C)$ satisfies \eqref{ineq:lowerGrowthEntire}.

Let $B(r)$ denote the open ball of radius $r$ which is centred at the origin.
We define the translation $f_a(z) \coloneqq f(z+a)$, for $a \in \C$,  which is an entire function with $f_a(0) = f(a)$.

We let $R >r$ and it follows from the Hausdorff-Young inequality (cf.~\cite{Kat76}) and \eqref{ineq:lowerGrowthEntire} that for $ a \in B(r)$
\begin{align*}
 \left( \sum_{n=0}^\infty  \left( \frac{\abs{D^n f(a)}}{n!} R^n  \right)^q \right)^{1/q}
 &=\left( \sum_{n=0}^\infty  \left( \frac{\abs{D^n f_a(0)}}{n!} R^n  \right)^q \right)^{1/q} \\
 &\leq M_p (f_a, R)
 \leq M_p (f, R + r)
 \leq c \frac{e^{R+r}}{(R + r)^{1/2p}}
\end{align*}
 where $q$ is the conjugate exponent of $p$ with $1/p + 1/q = 1$.

So it follows that
\begin{equation}  \label{ineq:DerivsWeightedBnd}
		\sum_{n=0}^\infty  \abs{D^n f(a)}^q  \frac{ R^{qn + q/2p}  (1 + r/R)^{q/2p} e^{-qR} } {c^q n!^q e^{qr}} \leq 1.
\end{equation}
We now apply Lemma~\ref{keyl} with $\alpha=q$ and $\beta=q/2p$ to obtain that there exists $C>0$ depending only on $r$ such that
\begin{equation*}
\frac{1}{m} \sum_{n=0}^m \abs{D^n f(a)}^q \leq C.
\end{equation*}

Next let $S(r) = \{ z \in \C \, : \, \abs{z} = r \}$ denote the sphere of radius $r$ centered at the origin, $k\in\N$ arbitrary, and let $a_1, \dotsc, a_k \in S(r)$.
We take the averages
\begin{equation*}
  \frac{1}{k}  \sum_{j=1}^k 	\frac{1}{m} \sum_{n=0}^m \abs{D^n f(a_j)}^q \leq C
\end{equation*}
and hence
\begin{equation*}
	\frac{1}{m} \sum_{n=0}^m  \left( \frac{1}{k}  \sum_{j=1}^k \abs{D^n f(a_j)}^q \right) \leq C.
\end{equation*}
Note that, if the $a_1, \dotsc, a_k \in S(r)$ are uniformly distributed,
\begin{equation*}
\lim_{k \to \infty}	 \frac{1}{k}  \sum_{j=1}^k \abs{D^n f(a_j)}^q = c_r M_q^q (D^n f, r)
\end{equation*}
where $c_r$ is a constant depending only on $r$.
So it follows that the average
\begin{equation*}
 \frac{1}{m} \sum_{n=0}^m  M_q^q (D^n f, r)
\end{equation*}
is bounded and thus $f$ cannot be a distributionally unbounded function for the differentiation operator.
\end{proof}

	Although Theorem \ref{thm:DIgrowthEntire} is an improvement on the previously known growth estimates, it naturally leads to the following question (originally posed in \cite[Problem 12]{BGB16}).
	\begin{qu}
		What is the critical order of growth of $D$-distributionally irregular entire functions?
	\end{qu}
	  Since $D$-hypercyclic and $D$-irregular entire functions share the same critical growth, one might expect that the optimal growth of $D$-distributionally irregular entire functions is similar to the frequently hypercyclic case.

\section{Growth of Irregular and Distributionally Irregular Harmonic Functions}  \label{sec:growthHarmonicFnsI}

We now turn our attention to the permissible growth of harmonic functions that are irregular and distributionally irregular with respect to partial differentiation operators.
We begin by recalling some notation and the background results relevant to this topic. Then we collect some auxiliary results (Subsection \ref{subsec:auxiliaryResults}) that we require in order to calculate the growth rates (Subsection \ref{subsec:GrowthHarmonic}).

The space $\harmsp$ of harmonic functions on $\R^N$, for $N \geq 2$, is a Fr\'{e}chet space when equipped
with the complete metric
\begin{equation*}
d(g,h) \coloneqq \sum_{n=1}^\infty 2^{-n} \frac{\abs{g-h}_{S(n)}}{1 + \abs{g-h}_{S(n)}}
\end{equation*}
for $g,h \in \harmsp$, and it corresponds to the topology of local uniform convergence.
Above we set $\vert f\vert_{S(n)} = \sup_{\vert x\vert = n} \vert f(x)\vert$
for $f \in \harmsp$.

We denote by $B(x, r)$ and $S(x, r)$, respectively, the open ball and the sphere of radius $r$ (in the euclidean metric) with centre at $x \in \R^N$.
When they are centred at the origin of $\R^N$ we simply write $B(r)$ and $S(r)$.

The sup-norm of $h \in \harmsp$ on $S(r)$ is defined as
\begin{equation*}
M_\infty(h,r) = \sup_{\norm{x} = r} \abs{h(x)}.
\end{equation*}
Let $\sigma_r$  denote the normalised $(N-1)$-dimensional measure on $S(r)$, so that $\sigma_r(S(r)) = 1$.
The  $L^2$-average of $h \in \harmsp$ on $S(r)$  is given by
\begin{equation*}
M_2 (h, r) = \left( \int_{S(r)} \abs{h}^2 \; \mathrm{d}\sigma_r \right)^{1/2}
\end{equation*}
where   $r > 0$ and the corresponding inner product  is defined by
\begin{equation*}
\langle g, h \rangle_r = \int_{S(r)} gh \; \mathrm{d}\sigma_r, \quad g, h \in \harmsp.
\end{equation*}

The following sharp growth rates of harmonic functions that are hypercyclic, with respect to partial differentiation operators, follow from the work of Aldred and Armitage~\cite{AA98a}.

\begin{enumerate}[label=(\Roman*)]

\item If $\varphi \colon \Rplus \to \Rplus$ is a function such that $\varphi(r) \to \infty$ as $r \to \infty$, then there exists a $\partial/\partial x_k$-hypercyclic function $h \in \harmsp$ with
\begin{equation*}
M_2(h,r) \leq \varphi(r) \frac{e^r}{r^{(N-1)/2}}
\end{equation*}
for $r>0$ sufficiently large.

\item Let $\alpha \in \N^N$.  There does not exist a $D^\alpha$-hypercyclic harmonic function $h \in \harmsp$ that satisfies
\begin{equation*}
M_2(h,r) \leq C \frac{e^r}{r^{(N-1)/2}}
\end{equation*}
for $r > 0$ and any constant $C > 0$.

\end{enumerate}

 Permissible growth of $\partial/\partial x_k$-frequently hypercyclic harmonic functions  was considered by Blasco et al.~\cite[Theorem 4.2]{BBGE10}, and subsequently the following optimal growth rates were identified in \cite{GST19}:
 for any constant $C > 0$ there exists
 a $\partial/\partial x_k$-frequently hypercyclic harmonic function $h \in \harmsp$ such that
 \begin{equation*}
 M_2(h,r) \leq C \frac{e^r}{r^{N/2 - 3/4}}, \quad  \textrm{ for all } r > 0.
 \end{equation*}

\subsection{Auxiliary Results} \label{subsec:auxiliaryResults}

We gather the notation and background from \cite{AA98a,AG01,ABR01} that is required for our investigation.

We denote by  $\harmpolys{m}$ the space of homogeneous harmonic polynomials  on $\mathbb{R}^N$ of
homogeneity degree $m \geq 0$.
The  harmonic analogue of the standard power series representation of
holomorphic functions states that any $h \in \harmsp$ has a unique expansion of the form
\begin{equation}  \label{eq:DecompHomoHarmPolys}
h = \sum_{m=0}^\infty H_m
\end{equation}
where $H_m \in \harmpolys{m}$ for each $m \ge 0$ and the expansion converges in the metric $d$, cf.~\cite[Corollary 5.34]{ABR01}.  Moreover, $\langle H_j, H_k \rangle_r = 0$ when
$j \neq k$, so  by orthogonality one has for any $r>0$ that the $L^2$-average of \eqref{eq:DecompHomoHarmPolys} is
\begin{equation*}
M_2^2(h,r) = \sum_{m=0}^\infty M_2^2 (H_m, r).
\end{equation*}
We denote the dimension of $\harmpolys{m}$ by $d_m = d_m(N)$ and it can be shown~\cite[Proposition 5.8]{ABR01} that $d_0 = 1$ for $N =2$ and
\begin{equation} \label{dim:dm}
d_m  = \frac{N+2m-2}{N+m-2}\binom{N+m-2}{m}
\end{equation}
for $N+m \geq 3$.
It follows easily from \eqref{dim:dm} that
\begin{equation}  \label{orderOfdm}
d_m = O(m^{N-2})
\end{equation}
as $m \to \infty$ (cf. \cite[p.107]{ABR01}).
Moreover, for $\alpha \in \N^N$ with $\abs{\alpha} = m$ and $H \in \harmpolys{m}$,  $D^\alpha H$ is constant and it follows from \cite[Lemma 1]{AA98a} that
\begin{equation}  \label{ineq:DHupperbound}
\abs{D^\alpha H} \leq m! \sqrt{d_m} r^{-m} M_2 \left( H, r \right)
\end{equation}
for $r>0$.
Further details on the spaces  $\harmpolys{m}$ can be found in   \cite[Chapter 5]{ABR01} and \cite[Chapter 2]{AG01}.

We also require an antiderivative for the partial differentiation operators $\partial/\partial x_k$ on $ \harmsp $, where $1 \leq k \leq N$.  Suitable linear maps were defined by
Aldred and Armitage~\cite{AA98a} by using a  specific orthogonal representation of harmonic polynomials constructed by Kuran~\cite{Kur71}.
We denote the $n^\mathrm{th}$ antiderivative, with respect to the coordinate $x_k$, by the linear map
\begin{equation}
P_{n,k} \colon \harmpolys{m} \to \harmpolys{m+n}  \label{defn:primitiveMaps}
\end{equation}
for $m, n \geq 0$.
For our purposes we do not require  to explicitly define the maps $P_{n,k}$, however we will utilise the pertinent properties which are contained in the following fundamental lemma
taken from \cite[Lemma 4]{AA98a}.
\begin{lemma} \label{lma:PrimitiveProps}
	Let $m, n \geq 0$, $N \geq 2$ and $1 \leq k \leq N$. If $H \in \harmpolys{m}$  then $P_{n,k}(H) \in \harmpolys{m+n}$,
	\begin{equation*}
	\frac{\partial^n}{\partial x_k^n} P_{n,k}(H) = H
	\end{equation*}
	and
	\begin{equation}  \label{ineq:primUpperBnd}
	M_2^2 (P_{n,k}(H), 1) \leq c_{n,m,N} M_2^2 (H,1)
	\end{equation}
	where
	\begin{equation*}
	c_{n,m,N} = \frac{(N+2m-2)!}{n!(N+2m+n-3)!(N+2m+2n-2)}.
	\end{equation*}
\end{lemma}

Similar to line (4.2) in  \cite{BBGE10}, for fixed $m$ we will use the simpler estimate
\begin{equation} \label{ineq:GrowthEst}
c_{n,m,N}  \leq \frac{c_m}{(n + m)!^2 (n + m +1)^{N -2}}
\end{equation}
for $n \in \N$,  where
\begin{equation}  \label{defn:c_m}
c_m = c_m(N) = (N+2m-2)!
\end{equation}
We  note that the different maps $P_{n,k}$ are mutually compatible since for $H \in \harmpolys{m}$ and $\ell, n \geq 0$ we have that
\begin{equation*}
P_{\ell + n,k} \left( H \right) = P_{\ell,k} \left( P_{n,k} \left( H \right) \right).
\end{equation*}
A proof of this fact can be found in \cite[Lemma 3.3]{GST19}. In particular it holds that $\frac{\partial^n}{\partial x_k^n} P_{\ell,k}(H)=P_{\ell-n,k}(H)$ for $\ell>n$.

We also need the following lemma on inequalities between $L^2$-norms of harmonic functions with respect to $N$-spheres with different centres.  This is essentially known, but for  completeness we include a proof.  The proof requires the Poisson integral and Harnack's inequality, which we recall below (full details can be found in \cite[Sections 1.3 and 1.4]{AG01}).

We let $\sigma_N = \sigma(S(1))$ be the surface area of the $N$-sphere $S(1)$, where $\sigma$ denotes the (unnormalised) surface area measure.
 The \emph{Poisson kernel} of the ball $B(x_0, r)$ is given by the function
\begin{equation*}
K_{x_0, r}(x,y) \coloneqq \frac{1}{\sigma_N r} \frac{r^2 - \norm{x - x_0}^2}{\norm{x-y}^N}
\end{equation*}
for $y \in S(x_0, r)$ and $x \in \R^N \setminus \{y\}$.

For a function $h$ continuous on $S(x_0,r)$, the \emph{Poisson integral} is defined as
\begin{equation*}
I_{h, x_0, r}(x) \coloneqq \int_{S(x_0,r)} K_{x_0, r}(x,y) \, h(y) \; \mathrm{d}\sigma(y)
\end{equation*}
for $x \in B(x_0, r)$.  It is a fundamental result of potential theory that $I_{h, x_0, r}$ defines a harmonic function on the ball $B(x_0, r)$ with boundary values on $S(x_0, r)$ given by $h$.

We recall that Harnack's inequality states if $h$ is a positive harmonic function on $B(x_0, r)$, then
\begin{equation*}
h(x) \leq \frac{\left( r + \norm{x - x_0} \right) r^{N-2}}{\left( r - \norm{x-x_0} \right)^{N-1}} h(x_0)
\end{equation*}
for each $x \in B(x_0, r)$.

\begin{lemma}\label{ineqharmnorm}
	Let $N \geq 2$.  Given $h \in \harmsp$, $r>0$, $R>r$, and $a\in\R^N$ with $\norm{a}\leq r$, we consider the translated harmonic function $h_a$ defined by $h_a(x)=h(a+x)$. We then have
	\[
	M_2(h_a,R)\leq C_N M_2(h,r+R),
	\]
	where $C_N>0$ is a constant that only depends on $N$.
\end{lemma}

\begin{proof}
	For brevity we denote the Poisson integrals of the subharmonic function $h^2$ on the spheres $S(a,R)$ and $S(r+R)$, respectively, by $I_{a, R}$ and $I_{0, r+R}$.

	First observe that
	\begin{align*}
	M_2^2(h,r+R) &= \int_{S(r+R)} \abs{h(y)}^2 \; \mathrm{d}\sigma_{(r+R)}(y) \\
	&=  \frac{1}{\sigma_N (r+R)^{N-1}}\int_{S(r+R)} h(y)^2 \; \mathrm{d}\sigma(y) = I_{0, r+R}(0)
	\intertext{and similarly}
	M_2^2(h_a,R) &= \int_{S(R)} \abs{h_a(y)}^2 \; \mathrm{d}\sigma_R(y)
	= \int_{S(a,R)} \abs{h(y)}^2 \; \mathrm{d}\sigma_R(y)  = I_{a, R}(a).
	\end{align*}

	Next notice	that $I_{a, R} = h^2$ on the sphere $S(a,R)$, and that the subharmonicity of $h^2$ gives that $h^2 \leq I_{0, r+R}$ in the ball $B(r+R)$.
	  So it holds that $I_{a, R} = h^2 \leq I_{0, r+R}$ on $S(a,R)$.
	By the maximum principle it follows that $I_{a, R}  \leq I_{0, r+R}$ on $B(a,R)$ and moreover
	\begin{equation*}
	I_{a, R}(a)  \leq I_{0, r+R}(a).
	\end{equation*}

	Next, since $I_{0, r+R}$ is a positive harmonic function on $B(r+R)$, we have by Harnack's inequality and the facts that $\norm{a} \leq r$  and $r<R$
	\begin{align*}
	I_{0, r+R}(a) &\leq \frac{\left( r+R + \norm{a} \right) (r+R)^{N-2}}{\left( r + R - \norm{a} \right)^{N-1}} I_{0, r+R}(0) \\
	&\leq \frac{\left( 2r+R \right) (r+R)^{N-2}}{R^{N-1}} I_{0, r+R}(0) \\
	&= \left( \frac{2r}{R} + 1\right) \left( \frac{r}{R} +1 \right)^{N-2} I_{0, r+R}(0) = C_N I_{0, r+R}(0)
	\end{align*}
	where the constant $C_N$ depends only on $N$ and the result follows.
\end{proof}

\subsection{Growth Rates} \label{subsec:GrowthHarmonic}

We are now ready to identify the optimal $M_2$-average growth of harmonic functions that are irregular with respect to partial differentiation. We note that the growth rates are the same as in the hypercyclic case.

\begin{theorem}   \label{thm:IrregularGrowth}
Let  $1 \leq k \leq N$.
\begin{enumerate}[label=(\roman*)]
\item Let $\varphi \colon \Rplus \to \Rplus$ be any function with $\varphi(r) \to \infty$ as $r \to \infty$.    Then there exists a $\partial/\partial x_k$-irregular harmonic
    function $h \in \harmsp$ with
\begin{equation*}
M_2(h,r) \leq \varphi(r) \frac{e^r}{r^{(N-1)/2}}
\end{equation*}
for $r>0$ sufficiently large.  \label{thm:Irreg1}

\item Let $\alpha \in \N^N$.  There does not exist a $D^\alpha$-irregular harmonic function $h \in \harmsp$ that satisfies
\begin{equation}\label{ineq:IrregNonExistRate}
M_2(h,r) \leq C \frac{e^r}{r^{(N-1)/2}}
\end{equation}
for $r > 0$ and any constant $C > 0$.  \label{thm:Irreg2}
\end{enumerate}
\end{theorem}

\begin{proof}
\ref{thm:Irreg1} This follows from \cite[Theorem 1]{AA98a} since hypercyclic vectors are irregular.

\noindent \ref{thm:Irreg2} Let $h \in \harmsp$.  We recall that the translation $h_a(x) \coloneqq h(x+a)$ preserves harmonicity and we further note that $h_a(0) = h(a)$, where $a
\in \R^N$.
Furthermore it follows from \eqref{eq:DecompHomoHarmPolys} that $h_a$ has a unique representation of the form
\begin{equation*}
h_a = \sum_{j=0}^\infty H_{a,j}
\end{equation*}
where $H_{a,j} \in \harmpolys{j}$.

For $n \in \N$ and $\alpha \in \N^N$ we may differentiate under the summation sign to obtain
\begin{equation}  \label{ineq:DiffatZero}
	D^{n\alpha} h(a) = D^{n\alpha} h_a(0) = \sum_{j=0}^\infty \left( D^{n\alpha} H_{a,j} \right) (0) = \left( D^{n\alpha} H_{a, n\abs{\alpha}} \right)(0)
\end{equation}
where we use the convention that $D^{n\alpha} = (D^\alpha)^n$.

Fix $r>0$ and let $a \in B(r)$.  For $R > r$ it follows from  \eqref{ineq:DiffatZero} and \eqref{ineq:DHupperbound} that
\begin{align*}
\abs{ D^{n\alpha} h(a) }   &\leq ( n\abs{\alpha})! \sqrt{d_{ n\abs{\alpha}}} R^{- n\abs{\alpha}} M_2(H_{a, n\abs{\alpha}}, R) \\
&\leq   ( n\abs{\alpha})! \sqrt{d_{n\abs{\alpha}}} R^{- n\abs{\alpha}} M_2(h_a, R).
\end{align*}
Applying Lemma \ref{ineqharmnorm} we get that
\begin{equation*}
M_\infty \left( D^{n\alpha} h, r \right)   \leq  c_N (n\abs{\alpha})! \sqrt{d_{n\abs{\alpha}}} R^{- n\abs{\alpha}} M_2(h, r+R).
\end{equation*}
Next suppose that \eqref{ineq:IrregNonExistRate} holds.  By \eqref{orderOfdm} we know that
\begin{equation*}
d_{n\abs{\alpha}} = O((n\abs{\alpha})^{N-2})
\end{equation*}
 as $n \to \infty$ and hence there exists a constant $C$, independent of $n$ and $r$, such that
\begin{equation*}
M_\infty \left( D^{n\alpha} h, r \right) \leq C \frac{(n\abs{\alpha})! (n\abs{\alpha})^{(N-2)/2} e^{r+R}}{R^{n\abs{\alpha}}  (r+R)^{(N-1)/2}}.
\end{equation*}
Applying Stirling's formula and choosing  $R = n\abs{\alpha} + (N-1)/2$ we obtain that
\begin{align*}
 M_\infty \left( D^{n\alpha} h, r \right)
 &\leq C \frac{(n\abs{\alpha})^{n\abs{\alpha} +(N-1)/2}  e^{r + n\abs{\alpha} + (N-1)/2}}{ \left( n\abs{\alpha} +(N-1)/2 \right)^{n\abs{\alpha} + (N-1)/2} e^{n\abs{\alpha}}  \left(1
 + \frac{r}{n\abs{\alpha} +(N-1)/2} \right)^{(N-1)/2}}  \\
 &\leq C  e^{r + (N-1)/2}  \left( 1 + \frac{N-1}{2n\abs{\alpha}} \right)^{-n\abs{\alpha}}
 \leq C  e^{r + (N-1)/2}.
\end{align*}
So we get that the sequence $\left\lbrace M_\infty  \left( D^{n\alpha} h, r \right) \right\rbrace_n$ is bounded and since $h$ does not have an unbounded orbit it  cannot be
irregular.
\end{proof}

Next we compute growth rates for $\partial/\partial x_k$-distributionally irregular harmonic functions and we note that the below growth is optimal.

\begin{theorem}  \label{thm:DIgrowth2}
Let  $1 \leq k \leq N$.
\begin{enumerate}[label=(\roman*)]
	\item    Let $\varphi \colon \Rplus \to \Rplus$ be any function with $\varphi(r) \to \infty$ as $r \to \infty$. Then there exists a harmonic function $h \in \harmsp$ which is
distributionally irregular with respect to the partial differentiation operator $\partial/\partial x_k$, such that
	\begin{equation}
	M_2(h,r) \leq \varphi(r) \frac{e^r}{r^{N/2 -3/4}}
	\end{equation}
	for $r>0$ sufficiently large.   \label{thm:DI01}

	\item  There does not exist a  $\partial/\partial x_k$-distributionally irregular $h \in \harmsp$ satisfying
	\begin{equation}
	M_2(h,r) \leq c \frac{e^r}{r^{N/2 -3/4}}    \label{ineq:DInonExistRate}
	\end{equation}
	where $c>0$ is constant and for $r>0$ sufficiently large.  \label{thm:DI02}
\end{enumerate}
\end{theorem}

\begin{proof}
\ref{thm:DI01}  Fix $1 \leq k \leq N$ and we  assume without loss of generality that $\inf_{r>0}\varphi (r)>0$.  We consider the  space
\begin{equation*}
Y \coloneqq \left\lbrace  h \in \harmsp \ : \ \lim_{r\to \infty} \frac{ M_2(h ,r)r^{N/2 -3/4}}{\varphi (r)e^r} =0 \right\rbrace
\end{equation*}
endowed with the corresponding sup-norm $\norm{\, \cdot \,}_Y$. Note that  $\left( Y,  \norm{\, \cdot \,}_Y \right)$  is a Banach space that is continuously embedded in $\harmsp$, endowed with the topology of local uniform convergence, and that every $h \in Y$ satisfies the desired growth condition.

We will apply Theorem \ref{thm:sufficient} and to this end we let $Y_0 = Y_1$ be the space of harmonic polynomials on $\R^N$.  The space $Y_0$ is dense in $Y$ and it follows
immediately that part \ref{item:suffA} of Theorem \ref{thm:sufficient} is satisfied.

For part \ref{item:suffB}, we use the antiderivative given in \eqref{defn:primitiveMaps} to define the map $S \colon Y_1 \to Y_1$
\begin{equation*}
	S \colon \sum_{j=0}^m H_j \mapsto \sum_{j=0}^m P_{1,k}(H_j)
\end{equation*}
where $H_j \in \harmpolys{j}$.
For all $H \in Y_1$ we have that $\partial/\partial x_k S H = H$ and since $\sum_{n=1}^\infty \partial^n/\partial x_k^n H$ is a finite sum it converges unconditionally.

To complete the proof it suffices to show that the series $\sum_{n=1}^\infty S^n H = \sum_{n=1}^\infty P_{n,k}(H)$ converges unconditionally in $Y$ for any polynomial $H \in \harmpolys{j}$, $j \geq 0$.
 (This was shown in  \cite[Theorem  4.2(a)]{BBGE10}, but for the convenience of the reader we outline the argument here.)

For a finite subset $F \subset \N$, we have by orthogonality, homogeneity, \eqref{ineq:primUpperBnd} and \eqref{ineq:GrowthEst} that

\begin{align*}
\norm{\sum_{n \in F} P_{n,k}(H)}_Y &= \sup_{r>0} \frac{r^{N/2 -3/4}}{\varphi (r)e^r} M_2 \left( \sum_{n \in F} P_{n,k}(H) ,r \right) \\
&= \sup_{r>0} \frac{r^{N/2 -3/4}}{\varphi (r)e^r} \left(  \sum_{n \in F} M_2^2 ( P_{n,k}(H) ,r)\right)^{1/2} \\
&\leq C \sup_{r>0} \frac{r^{N/2 -3/4}}{\varphi (r)e^r} \left(  \sum_{n \in F}  \frac{r^{2(n+j)}}{(n+j)!^2 (n+j+1)^{N-2}}  \right)^{1/2}.
\end{align*}
By an application of Lemma \ref{lma:BarnesEstimate}, it then follows that $\sum_{n = 0}^\infty P_{n,k}(H)$ converges unconditionally in $Y$.

\noindent\ref{thm:DI02}
Fix $1 \leq k \leq N$.
We consider the translated harmonic function $h_a(x) \coloneqq h(x+a)$ for $a \in \R^N$.
It follows from \eqref{eq:DecompHomoHarmPolys} that $h_a$ has a unique representation of the form
\begin{equation*}
h_a = \sum_{j=0}^\infty H_{a,j}
\end{equation*}
where $H_{a,j} \in \harmpolys{j}$.

 For $a \in S(r)$,   it follows from \eqref{ineq:DHupperbound} that
\begin{equation*}
		\abs{\frac{\partial^n}{\partial x_k^n} h(a) }  = 	\abs{\frac{\partial^n}{\partial x_k^n} h_a(0) }  \leq n! \sqrt{d_n} R^{-n} M_2(H_{a,n}, R) \end{equation*}
for any $R>r$. It then follows from Lemma~\ref{ineqharmnorm} that
\begin{equation*}
	\sum_{n=0}^\infty	\frac{R^{2n}}{n!^2{d_n}}\abs{\frac{\partial^n}{\partial x_k^n} h(a) }^2 \leq \sum_{n=0}^\infty  M_2^2(H_{a,n}, R)
=   M_2^2(h_a, R)\leq C_1 M_2^2(h,r+R)
\end{equation*}
where $C_1$ is a constant that depends only on $N$.


We recall that the sequence $(d_n)_n$  given by \eqref{dim:dm} is increasing and satisfies $d_n = O(n^{N-2})$ as $n \to \infty$.
So  applying our assumption we get that
\begin{equation*}
	\sum_{n=0}^\infty	\frac{R^{2n}}{n!^2{d_n}}\abs{\frac{\partial^n}{\partial x_k^n} h(a) }^2 \leq  C_2 \frac{e^{2(r+R)}}{(r+R)^{N  -3/2}}.
\end{equation*}
That is,
\begin{equation*}\label{ineq:sumha}
	\sum_{n=0}^\infty  \frac{R^{2n+N-3/2}}{n!^2n^{N-2}e^{2R}}\abs{ \frac{\partial^n}{\partial x_k^n} h(a) }^2 \leq C_2
\end{equation*}
for some constant $C_2$ that only depends on $r$ and $N$. By applying Lemma~\ref{keyl} for $\alpha=2$ and $\beta=N-3/2$, we obtain

\begin{equation*}
\frac{1}{l} \sum_{n=0}^l \abs{  \frac{\partial^n}{\partial x_k^n} h(a)}^2 \leq C
\end{equation*}
for every $l\in\N$ and $a\in B(r)$, and for some constant $C$ that only depends on $r$ and $N$. In particular, if we select a finite family $\{a_i \in S(r) \, : \,  i=0,\ldots,m\}$ uniformly distributed on $S(r)$ we get
\begin{equation*}
\frac{1}{m} \sum_{i=0}^m \left( \frac{1}{l} \sum_{n=0}^l \abs{  \frac{\partial^n}{\partial x_k^n} h(a_i)}^2\right) =
\frac{1}{l} \sum_{n=0}^l  \left( \frac{1}{m} \sum_{i=0}^m  \abs{  \frac{\partial^n}{\partial x_k^n} h(a_i)}^2\right) \leq C.
\end{equation*}
Taking limits as $m\to\infty$, we conclude that
\begin{equation*}
\frac{1}{l} \sum_{n=0}^l M_2^2\left(\frac{\partial^n}{\partial x_k^n} h,r\right) \leq C
\end{equation*}
for each $l\in\N$, and it follows that $h$ cannot be a distributionally unbounded harmonic function with respect to the operator $\partial/\partial x_k$.
\end{proof}

A natural further question arising from Theorem \ref{thm:DIgrowth2} is the following.
\begin{qu}
	What are the precise $M_p$-average growth rates of $\partial/\partial x_k$-distributionally irregular harmonic functions for $p \neq 2$?
\end{qu}

\section{Growth of $D^\alpha$-Distributionally Irregular Harmonic Functions} \label{sec:growthHarmonicFnsII}

In this section we study permissible growth rates, in terms of the sup-norm on spheres, of harmonic functions that are distributionally irregular with respect to general partial differentiation operators $D^\alpha$, for $\alpha \in \N^N$.  The hypercyclic case was considered by the investigation of Aldred and Armitage~\cite{AA98b} and the frequently hypercyclic case was studied by Blasco et al.~\cite{BBGE10}.

We begin by recalling some  notation and results from \cite{AA98a, AA98b} which are required in the sequel.  Set $c_2 =1$ and for $N \geq 3$ we define the constants
\begin{equation}  \label{defn:cN}
c_N = N \left( \prod_{j=1}^{N-1} \frac{(2j)^{2j}}{(2j +1)^{2j +1}} \right)^{1/2N}.
\end{equation}
It was shown in \cite[Section 3.2]{AA98b} that for $N \geq 3$
\begin{equation*}
c_N > \sqrt{\frac{N}{2}} \quad \text{ and } \quad c_N = \sqrt{\frac{N}{2}} + o(1), \text{ as } N \to \infty.
\end{equation*}

We also require suitable antiderivatives associated with the partial differentiation operators $D^\alpha$.
Using an inductive construction on the maps from \eqref{defn:primitiveMaps}, in \cite[Lemma 2]{AA98b} they identified linear maps
\begin{equation}  \label{defn:primitiveMapsAlpha}
P_{n,\alpha} \colon \harmpolys{m} \to \harmpolys{m+ n\abs{\alpha}}
\end{equation}
with the property that $D^{n\alpha} P_{n,\alpha} (H) = H$, for $H \in \harmpolys{m}$, where  $m,n \geq 0$.
Here we again use the convention that $D^{n\alpha} = (D^\alpha)^n$.

For our purposes, we do not need to explicitly define the maps $P_{n,\alpha}$, however we note that the different maps $P_{n,\alpha}$ are mutually compatible since for $H \in \harmpolys{m}$ and $\ell, n \geq 0$ we have that
\begin{equation*}
	P_{\ell + n,\alpha} \left( H \right) = P_{\ell,\alpha} \left( P_{n,\alpha} \left( H \right) \right)
\end{equation*}
and in particular it holds that $D^{n\alpha} P_{\ell,\alpha}(H)=P_{\ell-n,\alpha}(H)$ for $\ell>n$.

We will utilise the following estimate, which follows from \cite[Lemma 4]{AA98b}.
For $m, n \in \N$, $\alpha \in \N^N$ and $H \in \harmpolys{m}$,
\begin{equation}  \label{est:supPrimitiveH}
	M_\infty(P_{n,\alpha}(H), r) \leq C \frac{n^A \abs{\alpha}^A (m+1)^{(N-1)/2} (c_N r)^{n\abs{\alpha}}}{(n\abs{\alpha})!} M_\infty(H, r)
\end{equation}
for $r>0$, where $c_N$ is as defined in \eqref{defn:cN}, $A, C >0$ are constants depending only on $N$, and we may assume that $A \in \N$.

It follows from results in \cite{AA98b} that for nonzero $\alpha \in \R^N$, there exists a $D^\alpha$-hypercyclic harmonic function $h \in \harmsp$ such that for any $\varepsilon >0$, there is some $C_\varepsilon >0$ with
\begin{equation} \label{harmHcGrowthGeneralD}
M_\infty(h,r) \leq C_\varepsilon e^{(c_N + \varepsilon)r}.
\end{equation}
It also follows from \cite{AA98b} that  for $\alpha = (1,\dotsc,1)$, there does not exist a $D^\alpha$-hypercyclic $h \in \harmsp$ that satisfies
\begin{equation*}
M_\infty(h,r) \leq C e^{cr},
\end{equation*}
for any $c< \sqrt{N/2}$ and where $C>0$ is a constant.

In \cite{BBGE10} they strengthened growth condition \eqref{harmHcGrowthGeneralD} and extended it to the frequently hypercyclic case. In particular, they showed if $\varphi \colon \R_+ \to \R_+$ is a function such that $\varphi(r)/r^p \to \infty$, as $r \to \infty$,  for any $p \geq 0$, then there exists a $D^\alpha$-frequently hypercyclic harmonic function $h \in \harmsp$ such that
\begin{equation*}
M_\infty(h,r) \leq \varphi(r) e^{c_N r}
\end{equation*}
for $r>0$ sufficiently large.

Since hypercyclic vectors are irregular, we may immediately infer growth estimates for $D^\alpha$-irregular harmonic functions from the above results.	 It is also necessary to mention that Bayart and Ruzsa \cite{BR15} showed that there are frequently hypercyclic operators that are not distributionally chaotic, answering negatively Problem 36 in \cite{BBMP13}. This shows that, in general, we cannot deduce directly the growth estimates for distributional chaos from the corresponding estimates in \cite{BBGE10} for frequent hypercyclicity.
We give initial growth estimates in the distributionally irregular case in the following theorem.

\begin{theorem}  \label{thm:growthDalpha}
Let $N \geq 2$ and $\alpha \in \N^N$ with $\alpha \neq 0$.  Let $\varphi \colon \R_+ \to \R_+$ be a function such that $\varphi(r)/r^p \to \infty$, as $r \to \infty$,  for any $p \geq 0$.
Then there exists a harmonic function $h \in \harmsp$, distributionally irregular with respect to the differentiation operator $D^\alpha$, such that
\begin{equation*}
	M_\infty(h,r) \leq \varphi(r) e^{c_N r}
\end{equation*}
for $r >0$ sufficiently large and where $c_N$ is as given in \eqref{defn:cN}.
\end{theorem}

\begin{proof}
	The proof is similar to  that of Theorem \ref{thm:DIgrowth2}~\ref{thm:DI01}.
	We  assume without loss of generality that $\inf_{r>0}\varphi (r)>0$ and we consider the  space
	\begin{equation*}
	Y \coloneqq \left\lbrace  h \in \harmsp \ : \ \lim_{r\to\infty}  \frac{M_\infty(h,r)}{\varphi (r)e^{c_N r}} =0 \right\rbrace
	\end{equation*}
	endowed with the sup-norm $\norm{\, \cdot \,}_Y$. Note that  $\left( Y,  \norm{\, \cdot \,}_Y \right)$  is a Banach space which is continuously embedded in $\harmsp$ and that every $h \in Y$ satisfies the desired growth
	condition.

	We will apply Theorem \ref{thm:sufficient} and to this end we let $Y_0 = Y_1$ be the space of harmonic polynomials on $\R^N$.  Then the space $Y_0$ is dense in $Y$ and it follows
	immediately that part \ref{item:suffA} of Theorem \ref{thm:sufficient} is satisfied with respect to the operator $D^\alpha$.

	For part \ref{item:suffB}, we define the map $S \colon Y_1 \to Y_1$ by
	\begin{equation*}
	S \colon \sum_{j=0}^m H_j \mapsto \sum_{j=0}^m P_{1,\alpha}(H_j)
	\end{equation*}
	where $H_j \in \harmpolys{j}$ and the antiderivative $P_{1,\alpha}$ is as given in \eqref{defn:primitiveMapsAlpha}.
	For all $H \in Y_1$, note that $D^\alpha S H = H$, and since $\sum_{n=1}^\infty D^{n\alpha} H$ is a finite sum it converges unconditionally.

	To complete the proof it suffices to prove that the series $\sum_{n=1}^\infty S^n H$ converges unconditionally in $Y$ for any polynomial $H \in \harmpolys{m}$. (This calculation appears in \cite[Theorem  4.3]{BBGE10}, but for the convenience of the reader we outline the steps here.)

	Let $F \subset \N$ be finite. If $F \cap \{0, 1,\dotsc , L\} = \varnothing$, then by 	\eqref{est:supPrimitiveH}, the homogeneity of $H$, Lemma \ref{lma:BarnesEstimate}  it follows that
	\begin{align*}
		\norm{\sum_{n \in F} P_{n,\alpha}(H)}_Y
		&\leq \sup_{r>0} \frac{1}{\varphi(r) e^{c_N r}}  \sum_{n=L+1}^\infty C \frac{n^A \abs{\alpha}^A (m+1)^{(N-1)/2} (c_N r)^{n\abs{\alpha}}}{(n\abs{\alpha})!} M_\infty(H, r) \\
		&\leq C \sup_{r>0} \frac{1}{\varphi(r) e^{c_N r}} \sum_{n=L+1}^\infty \frac{n^A (c_N r)^{n\abs{\alpha}}}{(n\abs{\alpha})!} r^m \\
		&\leq C \sup_{r>0} \frac{r^{m+A}}{\varphi(r) e^{c_N r}} \sum_{n=L+1}^\infty \frac{(c_N r)^{n\abs{\alpha} - A}}{(n\abs{\alpha} -A)!} 		 \leq C \sup_{r>0} \frac{r^{m+A}}{\varphi(r)}
	\end{align*}
	where the constants $C>0$ above take different values.
	Using the assumption that $\varphi(r)/r^{m+A} \to \infty$ as $r \to \infty$, it follows that the series $\sum_{n=1}^\infty S^n H$ converges unconditionally in $Y$.

\end{proof}

The preceding discussion and theorem naturally give rise to the following question.
\begin{qu}
	What are the optimal  rates of growth of harmonic functions that are irregular and distributionally irregular with respect to $D^\alpha$?
\end{qu}
The analogous questions remains open for frequent hypercyclicity and even for hypercyclic harmonic functions (cf.~\cite[Section 6]{BBGE10}). Aldred and Armitage~\cite{AA98b} conjecture in the hypercyclic case that the sup-norm is of exponential type $\displaystyle \sqrt{N/2}$.

\section*{Acknowledgements}

This project was initiated during a visit by C. Gilmore to the Universitat Polit\`ecnica de Val\`encia and he wishes to thank the members of IUMPA for their hospitality and
mathematical stimulation during his visit.

We wish to thank Tom Carroll and Stephen Gardiner for helpful suggestions regarding Lemma \ref{ineqharmnorm}. We are also grateful to the anonymous referee for the careful reading and helpful suggestions that  improved considerably the presentation of this text.

%

\end{document}